\documentclass[12pt,reqno]{amsart}
\newtheorem{theorem}{Theorem}[section]
\newtheorem{definition}{Definition}[section]

\theoremstyle{definition}
\newtheorem{remark}{Remark}[section]
\newtheorem{example}{Example}[section]
\newtheorem{note}{Note}[section]
\newcommand{\be}{\begin{equation}}
\newcommand{\ee}{\end{equation}}
\newcommand{\bea}{\begin{eqnarray}}
\newcommand{\eea}{\end{eqnarray}}
\newcommand{\beb}{\begin{eqnarray*}}
\newcommand{\eeb}{\end{eqnarray*}}
\newcommand{\norm}[1]{\left\lVert#1\right\rVert}
\usepackage{amssymb,amsfonts,amsthm,setspace,indentfirst,mathrsfs}
\usepackage{minibox}
\usepackage{enumitem}
\usepackage{hyperref}
\numberwithin{equation}{section}

\begin{document}
\title[Rough $\mathcal{I}_2$-convergence in IFNS]{Rough ideal convergence of double sequences in intuitionistic fuzzy normed spaces}

\author[R. Mondal, N. Hossain]{Rahul Mondal$^1$, Nesar Hossain$^2$}

\address{$^{1}$Department of Mathematics, Vivekananda Satavarshiki Mahavidyalaya, Manikpara, Jhargram -721513, West Bengal, India.}
\address{$^2$Department of Mathematics, The University of Burdwan,  Burdwan - 713104, West Bengal, India.}

\email{imondalrahul@gmail.com$^1$; nesarhossain24@gmail.com$^2$}

\begin{abstract}
The idea of rough statistical convergence for double sequences was studied by \"{O}zcan and Or\cite{Ozcan} in a intuitionistic fuzzy normed space. Recently the same has been generalized  in the ideal context by Hossain and Banerjee\cite{Hossain2022} for sequences. Here in this paper we have discussed the idea of rough ideal convergence of double sequences in intuitionistic fuzzy normed spaces generalizing the idea of rough statistical convergence of double sequences. Also we have defined rough $\mathcal{I}_2$-cluster points for a double sequence and also investigated some of the basic properties associated with rough $\mathcal{I}_2$-limit set of a double sequence in a intuitionistic fuzzy normed space.
\end{abstract}
\maketitle
\section{Introduction}
 \noindent The concept of fuzzy sets was firstly introduced by Zadeh \cite{Zadeh} in $1965$ as a generalization of the concept of crisp set. A wide range of extensive applications in various branches of modern science and engineering can be found in \cite{Barros, Fradkov, Giles, Hong, Madore}. The idea of intuitionistic fuzzy sets was firstly given by Atanassov\cite{Atanassov} in $1986$ and later in $2004$, Park\cite{Park} introduced the concept of intuitionistic fuzzy metric spaces using this idea of  Atanassov. In 2006 Saadati and Park\cite{Saadati 2006} extended this concept to the theory of intuitionistic fuzzy normed spaces. Many authors have worked on intuitionistic fuzzy normed spaces\cite{Antal, Ozcan} using the idea of Saadati and Park\cite{Saadati 2006}.\\  
\indent After the introduction of the notion statistical convergence of sequences by Fast \cite{Fast} and Steinhaus \cite{Steinhaus} independently generalizing the the notion of ordinary convergence of sequences, Kostyrko et al. \cite{Kostyrko Salat Wilczynski} introduced two interesting extensions of this idea as $\mathcal{I}$ and $\mathcal{I}^*$-convergence of sequences using the structure of an ideal, formed by subsets of natural numbers. Several works by many authors in different directions can be found in \cite{Kumar, Karakaya, Konwar, Kirisci, Khan, Kisi, Mohiuddine} using the idea of \cite{Fast, Steinhaus, Kostyrko Salat Wilczynski}.\\
\indent In $2001$, Phu\cite{Phu2001} introduced the idea of rough convergence of sequences in finite dimensional normed linear spaces as a generalization of ordinary convergence of sequences. In \cite{Phu2001} he investigated some topological and geometrical properties of rough limit set of a sequence and introduced the idea of rough cauchy sequences. In 2003 this concept was extended to the infinite dimensional normed linear spaces by Phu\cite{Phu2003}. Later the idea of Phu\cite{Phu2001, Phu2003} was extended to rough statistical convergence using the concept of natural density by Ayter\cite{Ayter}. The concept of rough statistical convergence of sequences was extended to rough ideal convergence of sequences in 2013 by Pal et al.\cite{Pal2013}. Later many works\cite{Banerjee, Dundar, Esi, Ghosal, Gumus} have been done by many authors using the idea of Phu. \\
\indent In intuitionistic fuzzy normed spaces the idea of rough statistical convergence of sequences was defined by Antal et al. \cite{Antal} and later \"{O}zcan and Or\cite{Ozcan} studied the same notion in the setting of double sequences. Recently the idea of \"{O}zcan and Or\cite{Ozcan} has been generalized in the ideal context by Hossain and Banerjee\cite{Hossain2022} for sequences. In this paper we have generalized this concept in an ideal context and investigated some of the important results using the idea of \cite{Ozcan}.  
\section{Preliminaries}
Throughout $\mathbb{N}$ and $\mathbb{R}$ denote the set of natural numbers and set of reals respectively. First we recall some basic definitions and notations.
\begin{definition}\cite{Kostyrko Salat Wilczynski}
Let $X\neq \emptyset$. A class $\mathcal{I}$ of subsets of $X$ is said to be an ideal in $X$ if the following conditions hold:
\begin{enumerate}
    \item $\emptyset\in X$;
    \item $A,B\in \mathcal{I}\implies A\cup B\in \mathcal{I}$;
    \item $A\in\mathcal{I}$, $B\subset A\implies B\in\mathcal{I}$.
\end{enumerate}
\end{definition}
An ideal $\mathcal{I}$ is called non trivial if $X\notin  \mathcal{I}$ and proper if $\mathcal{I}\neq \{\emptyset\}$. A non trivial ideal is called admissible if $\{x\}\in\mathcal{I}$ for each $x\in X$.

\begin{definition}\cite{Kostyrko Salat Wilczynski}
    A non empty family $\mathcal{F}$ of subsets of a non empty set $X$ is said to be filter of $X$ if the following conditions hold:
    \begin{enumerate}
        \item $\emptyset \notin \mathcal{F}$;
        \item $A,B\in \mathcal{F}\implies A\cap B\in\mathcal{F}$;
        \item $A\in \mathcal{F}, A\subset B\implies B\in \mathcal{F}$.
    \end{enumerate}
\end{definition}
If $\mathcal{I}$ is a non trivial ideal of $X$ then the family $\mathcal{F(I)}=\{X\setminus A: A\in \mathcal{I}\}$ is a filter on $X$, called filter associated with the ideal $\mathcal{I}$. Throughout the paper $\mathcal{I}$ will stand for a non trivial admissible ideal in $\mathbb{N}$.

\begin{definition}\cite{Das 2008}
  A non trivial ideal $\mathcal{I}_2$ of $\mathbb{N}\times \mathbb{N}$ is called strongly admissible if $\{i\}\times \mathbb{N}$ and $\mathbb{N}\times \{i\}$ belong to $\mathcal{I}_2$ for each $i\in\mathbb{N}$.  
\end{definition}
It is clear that a strongly admissible ideal is admissible also. 

\begin{definition}
Let $K\subset \mathbb{N}$. Then the natural density of $K$ is denoted by $\delta(K)$ and is defined by $$\delta(K)=\lim_{n\to\infty}\frac{1}{n}|\{k\leq n: k\in K\}|,$$ provided the limit exists, where $|\cdot|$ designates the number of elements of the enclosed set. 
\end{definition}
It is clear that if $K$ is finite then $\delta (K)=0$.

\begin{definition}\cite{Mursaleen Edely}
The double natural density of the set $A\subset \mathbb{N}\times\mathbb{N}$ is defined by $$\delta_2(A)= \lim_{m,n\to\infty}\frac{|\{(i,j)\in A: i\leq m, j\leq n\}|}{mn}$$
  where $|\{(i,j)\in A: i\leq m, j\leq n\}|$ denotes the number of elements of $A$ not exceeding $m$ and $n$, respectively. It is clear that if $A$ is finite then $\delta_2(A)=0$. 
\end{definition}

\begin{definition}\cite{Das 2008}
A non trivial ideal $\mathcal{I}_2$ of $\mathbb{N}\times \mathbb{N}$ is said to be strongly admissible if $\{i\}\times \mathbb{N}$ and  $\mathbb{N}\times \{i\}$ belong to $\mathcal{I}_2$ for each $i\in \mathbb{N}$. 
\end{definition}
It is clear that a strongly admissible ideal is also admissible. Throughout the discussion $\mathcal{I}_2$ stands for an admissible ideal of $\mathbb{N}\times \mathbb{N}$.

\begin{definition}(see \cite{Das 2008})
  A double sequence $\{x_{mn}\}$ of real numbers is said to be convergent to $\xi\in\mathbb{R}$ if for any $\varepsilon>0$ there exists $\mathcal{N}_\varepsilon\in\mathbb{N}$ such that $|x_{mn}-\xi|<\varepsilon$ for all $m,n\geq \mathcal{N}_\varepsilon$. 
\end{definition}

\begin{definition}\cite{Mursaleen Edely}
  A double sequence $\{x_{mn}\}_{m,n\in\mathbb{N}}$ of real numbers is said to be statistically convergent to $\xi\in \mathbb{R}$ if for any $\varepsilon>0$, we have $\delta_2(A(\varepsilon))=0$ where $A(\varepsilon)=\{(m.n)\in\mathbb{N}\times \mathbb{N}: |x_{mn}-\xi|\geq \varepsilon\}$.  
\end{definition}

\begin{definition}(see \cite{Das 2008})
     A double sequence $\{x_{mn}\}_{m,n\in\mathbb{N}}$ of real numbers is said to be $\mathcal{I}_2$-convergent to $\xi\in \mathbb{R}$ if for every $\varepsilon>0$, the set $\{(m,n)\in\mathbb{N}\times\mathbb{N}: |x_{mn}-\xi|\geq \varepsilon \}\in \mathcal{I}_2$.
\end{definition}

\begin{remark}(see \cite{Das 2008})
$(a)$ If we take $\mathcal{I}_2=\mathcal{I}_2^0$, where $\mathcal{I}_2^0=\{A\subset \mathbb{N}\times\mathbb{N}: \exists\ m (A)\in\mathbb{N}: i,j\geq m (A)\implies (i,j)\notin A\}$, then $\mathcal{I}_2^0$ will be a non trivial strongly admissible ideal. In this case $\mathcal{I}_2$-convergence coincides with ordinary convergence of double sequences of real numbers.\\
$(b)$ If we take $\mathcal{I}_2=\mathcal{I}_2^\delta$, where $\mathcal{I}_2^\delta=\{A\subset \mathbb{N}\times \mathbb{N}: \delta_2(A)=0\}$, then $\mathcal{I}_2^\delta$-convergence becomes statistical convergence of double sequences of real numbers.
\end{remark}

Now, we recall some basic definitions and notations which will be useful in the sequal.

\begin{definition}\cite{Sklar}
A binary operation $\star : [0,1]\times [0,1]\rightarrow [0,1]$ is said to be a continuous $t$-norm if the following conditions hold:
\begin{enumerate}
    \item $\star$ is associative and commutative;
    \item $\star$ is continuous;
    \item $x\star 1=x$ for all $x\in [0,1]$;
    \item $x\star y\leq z\star w$ whenever $x\leq z$ and $y\leq w$ for each $x,y,z,w\in [0,1]$.
\end{enumerate}
\end{definition}

\begin{definition}\cite{Sklar}
A binary operation $\circ : [0,1]\times [0,1]\rightarrow [0,1]$ is said to be a continuous $t$-conorm if the following conditions are satisfied: 
\begin{enumerate}
    \item $\circ$ is associative and commutative;
    \item $\circ$ is continuous;
    \item $x\circ 0=x$ for all $x\in [0,1]$;
    \item $x\circ y\leq z\circ w$ whenever $x\leq z$ and $y\leq w$ for each $x,y,z,w\in [0,1]$.
\end{enumerate}
\end{definition}

\begin{example}\cite{Klement}
The following are the examples of $t$-norms:
\begin{enumerate}
    \item $x\star y=min\{x,y\}$;
    \item $x\star y=x.y$;
    \item $x\star y= max\{x+y-1,0\}$. This $t$-norm is known as Lukasiewicz $t$-norm.
\end{enumerate}
\end{example}

\begin{example}\cite{Klement}
The following are the examples of $t$-conorms:
\begin{enumerate}
    \item $x\circ y=max\{x,y\}$;
    \item $x\circ y=x+y-x.y$;
    \item $x\circ y=min\{x+y,1\}$. This is known as Lukasiewicz $t$-conorm.
\end{enumerate}
\end{example}

\begin{definition}\cite{Saadati 2006}
The $5$-tuple $(X,\mu,\nu,\star,\circ)$ is said to be an intuitionistic fuzzy normed space (in short, IFNS) if $X$ is a normed linear space, $\star$ is a continuous $t$-norm, $\circ$ is a continuous $t$-conorm and $\mu$ and $\nu$ are the fuzzy sets on $X\times (0,\infty)$ satisfying the following conditions for every $x,y\in X$ and $s,t>0$:
\begin{enumerate}
    \item $\mu(x,t)+\nu(x,t)\leq 1$;
    \item $\mu(x,t)>0$;
    \item $\mu(x,t)=1$ if and only if $x=0$;
    \item $\mu(\alpha x,t)=\mu(x,\frac{t}{|\alpha|})$ for each $\alpha\neq 0$;
    \item $\mu(x,t)\star \mu(y,s)\leq \mu(x+y,t+s)$;
    \item $\mu(x,t): (0,\infty)\rightarrow [0,1]$ is continuous in $t$;
    \item $\lim_{t\to\infty} \mu(x,t)=1$ and $\lim_{t\to 0}\mu(x,t)=0$;
    \item $\nu(x,t)<1$;
    \item $\nu(x,t)=0$ if and only if $x=0$;
    \item $\nu(\alpha x,t)=\nu(x,\frac{t}{|\alpha|})$ for each $\alpha\neq 0$;
    \item $\nu(x,t)\circ \nu(y,s)\geq \nu(x+y,s+t)$;
    \item $\nu(x,t): (0,\infty)\rightarrow [0,1]$ is continuous in $t$;
    \item $\lim_{t\to\infty} \nu(x,t)=0$ and $\lim_{t\to 0}\nu(x,t)=1$.
\end{enumerate}
In this case $(\mu,\nu)$ is called an intuitionistic fuzzy norm on $X$. 
\end{definition}

\begin{example}\cite{Saadati 2006}
Let $(X,\norm{\cdot})$ be a normed space. Denote $a\star b=ab$ and $a\circ b=min\{a+b,1\}$ for all $a,b\in [0,1]$ and let $\mu$ and $\nu$ be fuzzy sets on $X\times (0,\infty)$ defined as follows:\\
$$\mu(x,t)=\frac{t}{t+\norm{x}}, \  \nu(x,t)=\frac{\norm{x}}{t+\norm{x}}.$$ Then $(X,\mu,\nu,\star,\circ)$ is an intuitionistic fuzzy normed space.
\end{example}

\begin{definition}\cite{Saadati Park}
Let $(X,\mu,\nu,\star,\circ)$ be an IFNS with intuitionistic fuzzy norm $(\mu,\nu)$. For $r>0$, the open ball $B(x,\lambda,r)$ with center $x\in X$ and radius $0<\lambda<1$, is the set $$B(x,\lambda,r)=\{y\in X: \mu(x-y,r)>1-\lambda, \ \nu(x-y,r)<\lambda\}.$$
\end{definition}
Similarly,  closed ball is the set $\overline{B(x,\lambda,r)}=\{y\in X: \mu(x-y,r)\geq 1-\lambda, \ \nu(x-y,r)\leq \lambda\}$.

\begin{definition}\cite{Sen}\label{defi2.14}
Let $\{x_n\}_{n\in\mathbb{N}}$ be a sequence in an IFNS  $(X,\mu,\nu,\star,\circ )$. Then a point $\gamma\in X$ is called a $\mathcal{I}$-cluster point of $\{x_n\}_{n\in\mathbb{N}}$ with respect to the intuitionistic fuzzy norm $(\mu,\nu)$ if for every $\varepsilon>0$, $\lambda\in (0,1)$, the set $\{n\in\mathbb{N}: \mu(x_n-\gamma,\varepsilon)>1-\lambda \ \text{and}\ \nu(x_n-\gamma,\varepsilon)<\lambda \}\notin\mathcal{I}$. 
\end{definition}

\begin{definition}\cite{Mursaleen 2009}
     Let $\{x_{mn}\}$ be a double sequence in an IFNS $(X,\mu,\nu,\star,\circ )$. Then  $\{x_{mn}\}$ is said to be convergent to $\xi\in X$ with respect to the intuitionistic fuzzy norm $(\mu,\nu)$ if for every $\varepsilon>0$ and $\lambda\in(0,1)$ there exists $\mathcal{N}_\varepsilon\in \mathbb{N}$ such that $\mu(x_{mn}-\xi,\varepsilon)>1-\lambda$ and $\nu(x_{mn}-\xi,\varepsilon)<\lambda$ for all $m,n\geq \mathcal{N}_\varepsilon$. In this case we write $(\mu,\nu)\text{-}\lim x_{mn}=\xi$ or $x_{mn}\xrightarrow{(\mu,\nu)}\xi$.
\end{definition}

\begin{definition}\cite{Mursaleen 2010}
Let $\mathcal{I}_2$ be a non trivial ideal of $\mathbb{N}\times \mathbb{N}$ and $(X,\mu,\nu,\star,\circ)$ be an intuitionistic fuzzy normed space. A double sequence $x=\{x_{mn}\}$ of elements of $X$ is said to be $\mathcal{I}_2$-convergent to $L\in X$ if for each $\varepsilon>0$ and $t>0$, $\{(m,n)\in \mathbb{N}\times \mathbb{N}: \mu(x_{mn}-L,t)\leq 1-\varepsilon \ \text{or}\ \nu(x_{mn}-L,t)\geq \varepsilon\}\in \mathcal{I}_2$. In this case we write $\mathcal{I}_2^{(\mu,\nu)}\textit{-}\lim x=L$ or $x_{mn}\xrightarrow{\mathcal{I}_2^{(\mu,\nu)}}L$.
\end{definition}

\begin{definition}\cite{Antal}
Let $(X,\mu,\nu,\star,\circ)$ be an IFNS with intuitionistic fuzzy norm $(\mu,\nu)$. A sequence $\{x_n\}_{n\in\mathbb{N}}$ in $X$ is said to be rough statistical  convergent to $\xi\in X$ with respect to the norm $(\mu,\nu)$ for some non-negative number $r$ if for every $\varepsilon >0$ and $\lambda\in(0,1)$, $\delta(\{n\in\mathbb{N}: \mu(x_n-\xi,r+\varepsilon)\leq 1-\lambda \ \text{or}\ \nu(x_n-\xi,r+\varepsilon)\geq \lambda\})=0$.
\end{definition}

\begin{definition}\cite{Hossain2022}
Let $\{x_n\}_{n\in\mathbb{N}}$ be a sequence in an IFNS $(X,\mu,\nu,\star,\circ )$ and $r$ be a non-negative number. Then $\{x_n\}_{n\in\mathbb{N}}$ is said to be rough $\mathcal{I}$-convergent to $\xi\in X$ with respect to the intuitionistics fuzzy norm $(\mu,\nu)$ if for every $\varepsilon >0$ and $\lambda\in(0,1)$, $\{n\in\mathbb{N}: \mu(x_n-\xi,r+\varepsilon)\leq 1-\lambda \ \text{or}\ \nu(x_n-\xi,r+\varepsilon)\geq \lambda\}\in \mathcal{I}$. In this case $\xi$ is called $r\text{-}\mathcal{I}_{(\mu,\nu)}$-limit of $\{x_n\}_{n\in\mathbb{N}}$ and  we write $r\text{-}\mathcal{I}_{(\mu,\nu)}\text{-}\lim_{n\to\infty}x_n=\xi$ or $x_n\xrightarrow{r\text{-}\mathcal{I}_{(\mu,\nu)}}\xi$.
\end{definition}

\begin{definition}\cite{Ozcan}
    Let $\{x_{mn}\}$ be a double sequence in an IFNS $(X,\mu,\nu,\star,\circ )$ and $r$ be a non negative real number. Then $\{x_{mn}\}$ is said to be rough convergent (in short $r$-convergent) to $\xi\in X$ with respect to the intuitionistics fuzzy norm $(\mu,\nu)$ if for every $\varepsilon >0$ and $\lambda\in(0,1)$ there exists $\mathcal{N}_{\lambda}\in\mathbb{N}$ such that $\mu(x_{mn}-\xi,r+\varepsilon)>1-\lambda\ \text{and}\ \nu(x_{mn}-\xi,r+\varepsilon)<\lambda$  for all $m,n\geq \mathcal{N}_{\lambda}$. In this case we write $r_2^{(\mu,\nu)}-\lim x_{mn}=\xi \ \text{or}\ x_{mn}\xrightarrow{r_2^{(\mu,\nu)}}\xi$.
\end{definition}

\begin{definition}\cite{Ozcan}
    Let $\{x_{mn}\}$ be a double sequence in an IFNS $(X,\mu,\nu,\star,\circ )$ and $r$ be a non negative real number. Then $\{x_{mn}\}$ is said to be rough statistically convergent to $\xi\in X$ with respect to the intuitionistic fuzzy norm $(\mu,\nu)$ if for every $\varepsilon>0$ and $\lambda\in(0,1)$, $\delta_2(\{(m,n)\in\mathbb{N}\times \mathbb{N}: \mu(x_{mn}-\xi,r+\varepsilon)\leq 1-\lambda\ \text{or}\ \nu(x_{mn}-\xi,r+\varepsilon)\geq \lambda\})=0$. In this case we write $r\text{-}st_2^{(\mu,\nu)}\text{-}\lim x_{mn}=\xi\ \text{or}\ x_{mn}\xrightarrow{r-st_2^{(\mu,\nu)}}\xi$. 
\end{definition}

\section{Main Results}
 We first introduce the notion of rough ideal convergence of double sequences in an IFNS and then investigate some important results associated with rough $\mathcal{I}_2$-cluster points in the same space.

\begin{definition}\label{defi3.1}
   Let $\{x_{mn}\}$ be a double sequence in an IFNS $(X,\mu,\nu,\star,\circ )$ and $r$ be a non negative real number. Then $\{x_{mn}\}$ is said to be rough $\mathcal{I}_2$-convergent to $\xi\in X$ with respect to the intuitionistic fuzzy norm $(\mu,\nu)$ if for every $\varepsilon>0$ and $\lambda\in(0,1)$, $\{(m,n)\in\mathbb{N}\times \mathbb{N}: \mu(x_{mn}-\xi,r+\varepsilon)\leq 1-\lambda \ \text{or}\ \nu(x_{mn}-\xi,r+\varepsilon)\geq \lambda\}\in \mathcal{I}_2$. In this case $\xi$ is called $r\text{-}\mathcal{I}_2^{(\mu,\nu)}$-limit of $\{x_{mn}\}$ and  we write $r\text{-}\mathcal{I}_2^{(\mu,\nu)}\text{-}\lim x_{mn}=\xi\ \text{or}\ x_{mn}\xrightarrow{r-\mathcal{I}_2^{(\mu,\nu)}}\xi$.   
\end{definition}

\begin{remark}
    $(a)$ If we put $r=0$ in Definition \ref{defi3.1} then the notion of rough $\mathcal{I}_2$-convergence with respect to the intuitionistic fuzzy norm $(\mu,\nu)$ coincides with the notion of $\mathcal{I}_2$-convergence with respect to the intuitionistic fuzzy norm $(\mu,\nu)$. So, our main interest is on the fact $r>0$.\\
    $(b)$ If we use $\mathcal{I}_2=\mathcal{I}_2^0$ in Definition \ref{defi3.1}, then the notion of rough $\mathcal{I}_2$-convergence with respect to the intuitionistic fuzzy norm $(\mu,\nu)$ coincides with the notion of rough convergence of double sequences with respect to the intuitionistic fuzzy norm $(\mu,\nu)$.\\
    $(c)$ If we take $\mathcal{I}_2=\mathcal{I}_2^\delta$ in Definition \ref{defi3.1}, then the notion of rough $\mathcal{I}_2$-convergence with respect to the intuitionistic fuzzy norm $(\mu,\nu)$ coincides with the notion of rough statistical convergence of double sequences with respect to the intuitionistic fuzzy norm $(\mu,\nu)$.   
\end{remark}

\begin{note}
From Definition \ref{defi3.1}, we get $r\text{-}\mathcal{I}_2^{(\mu,\nu)}$-limit of $\{x_{mn}\}$ is not unique. So, in this regard we denote $\mathcal{I}_2^{(\mu,\nu)}\text{-}LIM_{x_{mn}}^r$ to mean the set of all $r\text{-}\mathcal{I}_2^{(\mu,\nu)}$-limit of $\{x_{mn}\}$, i.e., $\mathcal{I}_2^{(\mu,\nu)}\text{-}LIM_{x_{mn}}^r=\{\xi\in X: x_{mn}\xrightarrow{r-\mathcal{I}_2^{(\mu,\nu)}}\xi\}$. The double sequence $\{x_{mn}\}$ is called rough $\mathcal{I}_2$-convergent if $\mathcal{I}_2^{(\mu,\nu)}\text{-}LIM_{x_{mn}}^r\neq \emptyset$.
\end{note}

We denote $LIM_{x_{mn}}^{r_{(\mu,\nu)}}$ to mean the set of all rough convergent limits of the double sequence $\{x_{mn}\}$  with respect to the intuitionistic fuzzy norm $(\mu,\nu)$. The sequence $\{x_{mn}\}$ is called rough convergent if $LIM_{x_{mn}}^{r_{(\mu,\nu)}}\neq \emptyset$. If the sequence is unbounded then $LIM_{x_{mn}}^{r_{(\mu,\nu)}}=\emptyset$ \cite{Ozcan}, although in such cases $\mathcal{I}_2^{(\mu,\nu)}\text{-}LIM_{x_{mn}}^r\neq \emptyset$ may  happen which will be shown in the following example.

\begin{example}\label{exmp3.1}
    Let $(X,\norm{\cdot})$ be a real normed linear space and $\mu(x,t)=\frac{t}{t+\norm{x}}$ and $\nu(x,t)=\frac{\norm{x}}{t+\norm{x}}$ for all $x\in X$ and $t>0$. Also, let $a\star b=ab$ and $a\circ b=\min \{a+b,1\}$. Then $(X,\mu,\nu,\star,\circ)$ is an IFNS. Now let us consider ideal $\mathcal{I}_2$ consisting of all those subsets of $\mathbb{N}\times \mathbb{N}$ whose double natural density are zero. Let us consider the double sequence $\{x_{mn}\}$ by $x_{mn}=\begin{cases}
    (-1)^{m+n}, \ \text{if}\ m,n\neq i^2,i\in\mathbb{N}\\
    mn, \ \text{otherwise}
    \end{cases}.$ Then $\mathcal{I}_2^{(\mu,\nu)}\text{-}LIM_{x_{mn}}^r=\begin{cases}
        \emptyset, \ r<1\\
        [1-r,r-1],\ r\geq 1
    \end{cases}$ and $LIM_{x_{mn}}^{r_{(\mu,\nu)}}=\emptyset$ for any $r$.
\end{example}

\begin{remark}
   From Example \ref{exmp3.1}, we have $\mathcal{I}_2^{(\mu,\nu)}\text{-}LIM_{x_{mn}}^r\neq \emptyset$ does not imply that $LIM_{x_{mn}}^{r_{(\mu,\nu)}}\neq \emptyset$. But, whenever $\mathcal{I}_2$ is an admissible ideal then $LIM_{x_{mn}}^{r_{(\mu,\nu)}}\neq \emptyset $ implies $\mathcal{I}_2^{(\mu,\nu)}\text{-}LIM_{x_{mn}}^r\neq \emptyset$ as $\mathcal{I}_2^0\subset \mathcal{I}_2$.
\end{remark}

Now we define $\mathcal{I}_2$-bounded of double sequences in an IFNS analogue to (\cite{Ozcan}, Definition 3.6).
\begin{definition}
    Let $\{x_{mn}\}$ be a double sequence in an IFNS $(X,\mu,\nu,\star,\circ )$. Then $\{x_{mn}\}$ is said to be $\mathcal{I}_2$-bounded with respect to the intuitionistic fuzzy norm $(\mu,\nu)$ if for every $\lambda\in (0,1)$ there exists a positive real number $M$ such that $\{(m,n)\in\mathbb{N}\times \mathbb{N}: \mu(x_{mn},M)\leq 1-\lambda \ \text{or}\ \nu(x_{mn},M)\geq \lambda\}\in \mathcal{I}_2$.
\end{definition}

\begin{theorem}
Let $\{x_{mn}\}$ be a double sequence in an IFNS $(X,\mu,\nu,\star,\circ )$. Then $\{x_{mn}\}$ is $\mathcal{I}_2$-bounded if and only if $\mathcal{I}_2^{(\mu,\nu)}\text{-}LIM_{x_{mn}}^r\neq \emptyset$ for all $r>0$.
\end{theorem}

\begin{proof}
First suppose that $\{x_{mn}\}$ is $\mathcal{I}_2$-bounded in $X$ with respect to the intuitionistic fuzzy norm $(\mu,\nu)$. Then for every $\lambda\in (0,1)$ there exists a positive real number $M$ such that $\{(m,n)\in\mathbb{N}\times \mathbb{N}: \mu(x_{mn},M)\leq 1-\lambda \ \text{or}\ \nu(x_{mn},M)\geq \lambda\}\in \mathcal{I}_2$. Let $K=\{(m,n)\in\mathbb{N}\times \mathbb{N}: \mu(x_{mn},M)\leq 1-\lambda \ \text{or}\ \nu(x_{mn},M)\geq \lambda\}$. Now for $(i,j)\in K^c$, we have $\mu(x_{ij}-\theta, r+M)\geq \mu(x_{ij};M)\star \mu(\theta,r)>(1-\lambda)\star 1=1-\lambda$ and $\nu(x_{ij}-\theta, r+M)\leq \nu(x_{ij};M)\circ \nu(\theta,r)< \lambda\circ 0=\lambda$, where $\theta$ is the zero element of $X$. Therefore $\{(i,j)\in\mathbb{N}\times \mathbb{N}: \mu(x_{ij}-\theta,r+M)\leq 1-\lambda \ \text{or}\ \nu(x_{mn},M)\geq \lambda\}\subset K$. Since $K\in \mathcal{I}_2$, $\theta\in \mathcal{I}_2$. Hence $\mathcal{I}_2^{(\mu,\nu)}\text{-}LIM_{x_{mn}}^r\neq \emptyset$. 

Conversely, suppose that $\mathcal{I}_2^{(\mu,\nu)}\text{-}LIM_{x_{mn}}^r\neq \emptyset$. Then there exists $\beta\in \mathcal{I}_2^{(\mu,\nu)}\text{-}LIM_{x_{mn}}^r$ such that for every $\varepsilon>0$ and $\lambda\in (0,1)$ such that $\{(m,n)\in\mathbb{N}\times \mathbb{N}: \mu(x_{mn}-\beta,r+\varepsilon)\leq 1-\lambda \ \text{or}\ \nu(x_{mn}-\beta,r+\varepsilon)\geq \lambda\}\in \mathcal{I}_2$. This shows that almost all $x_{mn}$ are contained in some ball with center $\beta$. Hence $\{x_{mn}\}$ is $\mathcal{I}_2$-bounded. This completes the proof.
\end{proof}

Now we will discuss on some algebraic characterization of rough $\mathcal{I}_2$-convergence in an IFNS.
\begin{theorem}
    Let $\{x_{mn}\}$ and $\{y_{mn}\}$ be two double sequences in an IFNS $(X,\mu,\nu,\star,\circ )$. Then for some $r>0$, the following statements hold:
    \begin{enumerate}
        \item If $x_{mn}\xrightarrow{r-\mathcal{I}_2^{(\mu,\nu)}}\xi$ and $y_{mn}\xrightarrow{r-\mathcal{I}_2^{(\mu,\nu)}}\eta$, then $x_{mn}+y_{mn}\xrightarrow{r-\mathcal{I}_2^{(\mu,\nu)}}\xi+\eta$.
        \item  If $x_{mn}\xrightarrow{r-\mathcal{I}_2^{(\mu,\nu)}}\xi$ and $k(\neq 0)\in\mathbb{R}$, then  $kx_{mn}\xrightarrow{r-\mathcal{I}_2^{(\mu,\nu)}}k\xi$.
    \end{enumerate}
\end{theorem}

\begin{proof}
    Let $\{x_{mn}\}$ and $\{y_{mn}\}$ be two double sequences in an IFNS $(X,\mu,\nu,\star,\circ )$, $r>0$ and $\lambda\in (0,1)$.
    \begin{enumerate}
        \item Let $x_{mn}\xrightarrow{r-\mathcal{I}_2^{(\mu,\nu)}}\xi$ and $y_{mn}\xrightarrow{r-\mathcal{I}_2^{(\mu,\nu)}}\eta$. Also, let $\varepsilon>0$ be given. Now, for a given $\lambda\in(0,1)$, choose $s\in (0,1)$ such that $(1-s)\star (1-s)>1-\lambda$ and $s\circ s<\lambda$. Then $A,B \in \mathcal{I}_2$, where $A= \{(m,n)\in\mathbb{N}\times \mathbb{N}: \mu(x_{mn}-\xi,\frac{r+\varepsilon}{2})\leq 1-s \ \text{or}\ \nu(x_{mn}-\xi,\frac{r+\varepsilon}{2})\geq s\}$ and $B= \{(m,n)\in\mathbb{N}\times \mathbb{N}: \mu(y_{mn}-\eta,\frac{r+\varepsilon}{2})\leq 1-s \ \text{or}\ \nu(y_{mn}-\eta,\frac{r+\varepsilon}{2})\geq s\}$. So, $A^c\cap B^c\in \mathcal{F(I_\text{2})}$. Now for $(i,j)\in A^c\cap B^c$, we have $\mu(x_{ij}+y_{ij}-(\xi+\eta),r+\varepsilon)\geq \mu(x_{ij}-\xi,\frac{r+\varepsilon}{2} )\star\mu(y_{ij}-\eta,\frac{r+\varepsilon}{2})>(1-s)\star(1-s)>1-\lambda$ and $\nu(x_{ij}+y_{ij}-(\xi+\eta),r+\varepsilon)\leq \nu(x_{ij}-\xi,\frac{r+\varepsilon}{2} )\circ \nu(y_{ij}-\eta,\frac{r+\varepsilon}{2})<s\circ s<\lambda$. Therefore $\{(i,j)\in \mathbb{N}\times \mathbb{N}:\mu(x_{ij}+y_{ij}-(\xi+\eta),r+\varepsilon)\leq 1-\lambda \ \text{or}\  \nu(x_{ij}-\xi,\frac{r+\varepsilon}{2} )\geq \lambda\}\subset A\cup B$. Since $A\cup B\in \mathcal{I}_2$, $\{(i,j)\in \mathbb{N}\times \mathbb{N}:\mu(x_{ij}+y_{ij}-(\xi+\eta),r+\varepsilon)\leq 1-\lambda \ \text{or}\  \nu(x_{ij}-\xi,\frac{r+\varepsilon}{2} )\geq \lambda\}\in\mathcal{I}_2$. Therefore $x_{mn}+y_{mn}\xrightarrow{r-\mathcal{I}_2^{(\mu,\nu)}}\xi+\eta$.
        \item Let $x_{mn}\xrightarrow{r-\mathcal{I}_2^{(\mu,\nu)}}\xi$ and $k(\neq 0)\in\mathbb{R}$. Then, $\{(m,n)\in\mathbb{N}\times \mathbb{N}: \mu(x_{mn}-\xi,\frac{r+\varepsilon}{|k|})\leq 1-\lambda \ \text{or}\ \nu(x_{mn}-\xi,\frac{r+\varepsilon}{|k|})\geq \lambda\}\in \mathcal{I}_2$. Therefore, $\{(m,n)\in\mathbb{N}\times \mathbb{N}: \mu(kx_{mn}-k\xi,r+\varepsilon)\leq 1-\lambda \ \text{or}\ \nu(kx_{mn}-k\xi,r+\varepsilon)\geq \lambda\}\in \mathcal{I}_2$. Hence $kx_{mn}\xrightarrow{r-\mathcal{I}_2^{(\mu,\nu)}}k\xi$. This completes the proof.
    \end{enumerate}
\end{proof}

Now we prove some topological and geometrical properties of the set $\mathcal{I}_2^{(\mu,\nu)}\text{-}LIM_{x_{mn}}^r$. 
\begin{theorem}\label{thm3.3}
     Let $\{x_{mn}\}$ be a double sequence in an IFNS $(X,\mu,\nu,\star,\circ )$. Then for all $r>0$, the set $\mathcal{I}_2^{(\mu,\nu)}\text{-}LIM_{x_{mn}}^r$ is closed.
\end{theorem}

\begin{proof}
    If $\mathcal{I}_2^{(\mu,\nu)}\text{-}LIM_{x_{mn}}^r=\emptyset$ then there is nothing to prove. So, let $\mathcal{I}_2^{(\mu,\nu)}\text{-}LIM_{x_{mn}}^r\neq \emptyset$.   Suppose thet $\{z_{mn}\}$ is a double sequence in $\mathcal{I}_2^{(\mu,\nu)}\text{-}LIM_{x_{mn}}^r$ such that $z_{mn}\xrightarrow{(\mu,\nu)}y_0$. Now, for a given $\lambda\in (0,1)$, choose $s\in (0,1)$ such that $(1-s)\star (1-s)>1-\lambda$ and $s\circ s<\lambda$. Let $\varepsilon>0$ be given. Then there exists $m_0\in \mathbb{N}$ such that $\mu(z_{mn}-y_0,\frac{\varepsilon}{2})>1-s$ and $\nu(z_{mn}-y_0,\frac{\varepsilon}{2})<s$ for all $m,n\geq m_0$. Suppose $i,j>m_0$. Then $\mu(z_{ij}-y_0,\frac{\varepsilon}{2})>1-s$ and $\nu(z_{ij}-y_0,\frac{\varepsilon}{2})<s$. Also, $P=\{(m,n)\in \mathbb{N}\times\mathbb{N}: \mu(x_{mn}-z_{ij},r+\frac{\varepsilon}{2})\leq 1-s\ \text{or}\ \nu(x_{mn}-z_{ij},r+\frac{\varepsilon}{2})\geq s \} \in\mathcal{I}_2$. Now, for $(p,q)\in P^c$, we have $\mu(x_{pq}-y_0,r+\varepsilon)\geq \mu(x_{pq}-z_{ij},r+\frac{\varepsilon}{2})\star\mu(z_{ij}-y_0,\varepsilon)>(1-s)\star (1-s)>1-\lambda$ and $\nu(x_{pq}-y_0,r+\varepsilon)\leq \nu(x_{pq}-z_{ij},r+\frac{\varepsilon}{2})\circ \nu(z_{ij}-y_0,\frac{\varepsilon}{2})<s\circ s<\lambda$. Therefore $\{(m,n)\in \mathbb{N}\times\mathbb{N}: \mu(x_{mn}-y_0,r+\frac{\varepsilon}{2})\leq 1-s\ \text{or}\ \nu(x_{mn}-y_0,r+\frac{\varepsilon}{2})\geq s \}\subset P$. Since $P\in \mathcal{I}_2$, $y_0\in \mathcal{I}_2^{(\mu,\nu)}\text{-}LIM_{x_{mn}}^r$. Therefore $\mathcal{I}_2^{(\mu,\nu)}\text{-}LIM_{x_{mn}}^r$ is closed. This completes the proof.
\end{proof}

\begin{theorem}
    Let $\{x_{mn}\}$ be a double sequence in an IFNS $(X,\mu,\nu,\star,\circ )$. Then for all $r>0$, the set $\mathcal{I}_2^{(\mu,\nu)}\text{-}LIM_{x_{mn}}^r$ is convex.
\end{theorem}

\begin{proof}
    Let $x_1,x_2\in \mathcal{I}_2^{(\mu,\nu)}\text{-}LIM_{x_{mn}}^r$ and $\kappa
    \in (0,1)$. Let $\lambda\in (0,1)$. Choose $s\in (0,1)$ such that $(1-s)\star (1-s)>1-\lambda$ and $s\circ s<\lambda$. Then for every $\varepsilon>0$, the sets $H,T\in \mathcal{I}_2$ where $H=\{(m,n)\in\mathbb{N}\times\mathbb{N}: \mu(x_{mn}-x_1,\frac{r+\varepsilon}{2(1-\kappa)})\leq 1-s\ \text{or}\ \nu(x_{mn}-x_1,\frac{r+\varepsilon}{2(1-\kappa)})\geq s \}$ and $T=\{(m,n)\in\mathbb{N}\times\mathbb{N}: \mu(x_{mn}-x_2,\frac{r+\varepsilon}{2\kappa})\leq 1-s\ \text{or}\ \nu(x_{mn}-x_2,\frac{r+\varepsilon}{2\kappa})\geq s \}$. Now for $(m,n)\in H^c\cap T^c$, we have $\mu(x_{mn}-[(1-\kappa)x_1+\kappa x_2],r+\varepsilon)\geq \mu((1-\kappa)(x_{mn}-x_1),\frac{r+\varepsilon}{2})\star \mu(\kappa(x_{mn}-x_2),\frac{r+\varepsilon}{2})=\mu(x_{mn}-x_1,\frac{r+\varepsilon}{2(1-\kappa)})\star \mu(x_{mn}-x_2,\frac{r+\varepsilon}{2\kappa})>(1-s)\star (1-s)>1-\lambda$ and $\nu(x_{mn}-[(1-\kappa)x_1+\kappa x_2],r+\varepsilon)\leq  \nu((1-\kappa)(x_{mn}-x_1),\frac{r+\varepsilon}{2})\circ \nu(\kappa(x_{mn}-x_2),\frac{r+\varepsilon}{2})=\nu(x_{mn}-x_1,\frac{r+\varepsilon}{2(1-\kappa)})\circ \nu(x_{mn}-x_2,\frac{r+\varepsilon}{2\kappa})< s\circ s<\lambda$, which gives that $\{(m,n)\in\mathbb{N}\times\mathbb{N}:\mu(x_{mn}-[(1-\kappa)x_1+\kappa x_2],r+\varepsilon)\leq 1-\lambda \ \text{or}\ \nu(x_{mn}-[(1-\kappa)x_1+\kappa x_2],r+\varepsilon)\geq \lambda \}\subset H\cup T$. Since $H\cup T\in \mathcal{I}_2$, $(1-\kappa)x_1+\kappa x_2\in \mathcal{I}_2^{(\mu,\nu)}\text{-}LIM_{x_{mn}}^r$. Therefore  $\mathcal{I}_2^{(\mu,\nu)}\text{-}LIM_{x_{mn}}^r$ is convex. This completes the proof.
\end{proof}

\begin{theorem}
    A double sequence $\{x_{mn}\}$ in an IFNS $(X,\mu,\nu,\star,\circ )$ is rough $\mathcal{I}_2$-convergent to $\beta\in X$ with respect to the intuitionistic fuzzy normed spaces $(\mu,\nu)$ for some $r>0$ if there exists a double sequence $\{y_{mn}\}$ in $X$ such that $y_{mn}\xrightarrow{\mathcal{I}_2^{(\mu,\nu)}}\beta$ and for every $\lambda\in (0,1)$, $\mu(x_{mn}-y_{mn},r)>1-\lambda$ and $\nu(x_{mn}-y_{mn},r)<\lambda$ for all $m,n\in \mathbb{N}$.
\end{theorem}

\begin{proof}
    Let $\varepsilon>0$ be given. Now, for a given $\lambda\in(0,1)$ choose $s\in(0,1)$ such that $(1-s)\star (1-s)>1-\lambda$ and $s\circ s<\lambda$. Supose that $y_{mn}\xrightarrow{\mathcal{I}_2^{(\mu,\nu)}}\beta$ and $\mu(x_{mn}-y_{mn},r)>1-s$ and $\nu(x_{mn}-y_{mn},r)<s$ for all $m,n\in \mathbb{N}$. Then the set $P=\{(m,n)\in \mathbb{N}\times \mathbb{N}: \mu(y_{mn}-\beta,\varepsilon)\leq 1-s \ \text{or}\ \nu(y_{mn}-\beta,\varepsilon)\geq s\}\in \mathcal{I}_2$. Now for $(i,j)\in P^c$, we have $\mu(x_{ij}-\beta,r+\varepsilon)\geq \mu(x_{ij}-y_{ij},r)\star \mu(y_{ij}-\beta,\varepsilon)>(1-s)\star (1-s)>1-\lambda$ and $\nu(x_{ij}-\beta,r+\varepsilon)\leq \nu(x_{ij}-y_{ij},r)\circ \nu(y_{ij}-\beta,\varepsilon)<s\circ s<\lambda$. Therefore $\{(i,j)\in\mathbb{N}\times\mathbb{N}: \mu(x_{ij}-\beta,r+\varepsilon)\leq 1-\lambda \ \text{or}\ \nu(x_{ij}-\beta,r+\varepsilon)\geq \lambda \}\subset P$. Since $P\in \mathcal{I}_2$, $\{(i,j)\in\mathbb{N}\times\mathbb{N}: \mu(x_{ij}-\beta,r+\varepsilon)\leq 1-\lambda \ \text{or}\ \nu(x_{ij}-\beta,r+\varepsilon)\geq \lambda \}\in\mathcal{I}_2$. Therefore  $\{x_{mn}\}$ is rough $\mathcal{I}_2$-convergent to $\beta$ with respect to the probabilistic norm $(\mu,\nu)$. This completes the proof.
\end{proof}

\begin{theorem}
    Let $\{x_{mn}\}$ be a double sequence in an IFNS $(X,\mu,\nu,\star,\circ )$. Then there do not exist $\beta_1,\beta_2\in \mathcal{I}_2^{(\mu,\nu)}\text{-}LIM_{x_{mn}}^r$ for some $r>0$ and every $\lambda\in (0,1)$ such that $\mu(\beta_1-\beta_2,mr)\leq 1-\lambda$ and $\nu(\beta_1-\beta_2,mr)\geq \lambda$ for $m(\in\mathbb{R})>2$.
\end{theorem}

\begin{proof}
    We prove it by contradiction. If possible, let there exists $\beta_1,\beta_2\in \mathcal{I}_2^{(\mu,\nu)}\text{-}LIM_{x_{mn}}^r$ such that $\mu(\beta_1-\beta_2,mr)\leq 1-\lambda$ and $\nu(\beta_1-\beta_2,mr)\geq \lambda$ for $m(\in\mathbb{R})>2$. Now, for a given $\lambda\in (0,1)$ choose $s\in (0,1)$ such that $(1-s)\star (1-s)>1-\lambda$ and $s\circ s<\lambda$. Then for every $\varepsilon>0$, the sets $A,B\in\mathcal{I}_2$ where $A=\{(m,n)\in\mathbb{N}\times\mathbb{N}: \mu(x_{mn}-\beta_1,r+\frac{\varepsilon}{2})\leq 1-s \ \text{or}\ \nu(x_{mn}-\beta_1,r+\frac{\varepsilon}{2})\geq s \}\in\mathcal{I}_2 $ and $B=\{(m,n)\in\mathbb{N}\times\mathbb{N}: \mu(x_{mn}-\beta_2,r+\frac{\varepsilon}{2})\leq 1-s \ \text{or}\ \nu(x_{mn}-\beta_2,r+\frac{\varepsilon}{2})\geq s \}\in\mathcal{I}_2 $. Then $A^c\cap B^c\in \mathcal{F(I_\text{2})}$. Now for $(m,n)\in A^c\cap B^c$, we have $\mu(\beta_1-\beta_2,2r+\varepsilon)\geq \mu (x_{mn}-\beta_1,r+\frac{\varepsilon}{2})\star \mu(x_{mn}-\beta_2,r+\frac{\varepsilon}{2})>(1-s)\star (1-s)>1-\lambda$ and $\nu(\beta_1-\beta_2,2r+\varepsilon)\leq \nu(x_{mn}-\beta_1,r+\frac{\varepsilon}{2})\circ \nu(x_{mn}-\beta_2,r+\frac{\varepsilon}{2})<s\circ s<\lambda$. Therefore, \begin{equation}\label{eqn1}
      \mu(\beta_1-\beta_2,2r+\varepsilon)>1-\lambda \ \text{and}\  \nu(\beta_1-\beta_2,2r+\varepsilon)<\lambda 
    \end{equation} Now if we put $\varepsilon=mr-2r, m>2$ in Equation \ref{eqn1} then we have $\mu(\beta_1-\beta_2,mr)>1-\lambda \ \text{and}\  \nu(\beta_1-\beta_2,mr)<\lambda$, which is a contradiction. This completes the proof.
\end{proof}

Now we define $\mathcal{I}_2$-cluster point analogue to Definition \ref{defi2.14}. \"{O}zcan and Or \cite{Ozcan} defined rough statistical cluster point of double sequences in an IFNS and, here, we give its ideal version in the same sapce. Also, we prove an important result analogue to (\cite{Sen}, Theorem 4.7) in the same space which will be useful in the sequal. 
\begin{definition}
    Let $\{x_{mn}\}$ be a double sequence in an IFNS $(X,\mu,\nu,\star,\circ )$. Then a point $\zeta\in X$ is said to be $\mathcal{I}_2$-cluster point of  $\{x_{mn}\}$ with respect to the intuitionistic fuzzy norm $(\mu,\nu)$ if for every $\varepsilon>0$ and $\lambda\in (0,1)$, $\{(m,n)\in \mathbb{N}\times \mathbb{N}: \mu(x_{mn}-\zeta,\varepsilon)>1-\lambda \ \text{and}\ \nu(x_{mn}-\zeta,\varepsilon)<\lambda\}\notin\mathcal{I}_2$. 
\end{definition}
We denote $\Gamma_{({x_{mn}})}(\mathcal{I}_2^{(\mu,\nu)})$ to mean the set of all $\mathcal{I}_2$-cluster points of $\{x_{mn}\}$ with respect to the intuitionistic fuzzy norm $(\mu,\nu)$ .
\begin{definition}\label{defi3.4}
    Let $\{x_{mn}\}$ be a double sequence in an IFNS $(X,\mu,\nu,\star,\circ )$ and $r\geq 0$. Then a point $\beta\in X$ is said to be rough $\mathcal{I}_2$-cluster point of  $\{x_{mn}\}$ with respect to the intuitionistic fuzzy norm $(\mu,\nu)$ if for every $\varepsilon>0$ and $\lambda\in (0,1)$, $\{(m,n)\in \mathbb{N}\times \mathbb{N}: \mu(x_{mn}-\beta,r+\varepsilon)>1-\lambda \ \text{and}\ \nu(x_{mn}-\beta,r+\varepsilon)<\lambda\}\notin\mathcal{I}_2$. 
\end{definition}
We denote $\Gamma_{({x_{mn}})}^r(\mathcal{I}_2^{(\mu,\nu)})$ to mean the set of all rough $\mathcal{I}_2$-cluster points of $\{x_{mn}\}$ with respect to the intuitionistic fuzzy norm $(\mu,\nu)$ .

\begin{remark}
    Now if we put $r=0$ in Definition \ref{defi3.4}, then $\Gamma_{({x_{mn}})}^r(\mathcal{I}_2^{(\mu,\nu)})=\Gamma_{({x_{mn}})}(\mathcal{I}_2^{(\mu,\nu)})$.
\end{remark}

\begin{theorem}\label{thm3.7}
    Let $\{x_{mn}\}$ be a double sequence in an IFNS $(X,\mu,\nu,\star,\circ )$ such that $x_{mn}\xrightarrow{\mathcal{I}_2^{(\mu,\nu)}}L$. Then $\Gamma_{({x_{mn}})}(\mathcal{I}_2^{(\mu,\nu)})=\{L\}$.
\end{theorem}

\begin{proof}
    If possible let $\Gamma_{({x_{mn}})}(\mathcal{I}_2^{(\mu,\nu)})=\{L, \mathcal{J}\}$, where $L\neq \mathcal{J}$. For a given $\lambda\in (0,1)$, choose $s\in (0,1)$ such that $(1-s)\star (1-s)>1-\lambda$ and $(s)\circ (s)<\lambda$. Then for every $\varepsilon>0$,  $K_1= \{(m,n)\in \mathbb{N}\times \mathbb{N}: \mu(x_{mn}-L,\frac{\varepsilon}{2})>1-s \ \text{and}\ \nu(x_{mn}-L,\frac{\varepsilon}{2})<s\}\notin\mathcal{I}_2$ and $K_2=\{(m,n)\in \mathbb{N}\times \mathbb{N}: \mu(x_{mn}-\mathcal{J},\frac{\varepsilon}{2})>1-s \ \text{and}\ \nu(x_{mn}-\mathcal{J},\frac{\varepsilon}{2})<s\}\notin\mathcal{I}_2$. Clearly $K_1\cap K_2=\emptyset$, If not, let $(i,j)\in K_1\cap K_2$. Then $\mu (L-\mathcal{J},\varepsilon)\geq \mu(x_{ij}-L,\frac{\varepsilon}{2})\star \mu(x_{ij}-\mathcal{J},\frac{\varepsilon}{2})>(1-s)\star (1-s)>1-\lambda$ and $\nu(L-\mathcal{J},\varepsilon)\leq \nu(x_{ij}-L,\frac{\varepsilon}{2})\circ \nu(x_{ij}-\mathcal{J},\frac{\varepsilon}{2})<s\circ s<\lambda$. Since $\lambda\in (0,1)$ is arbitrary, $\mu (L-\mathcal{J},\varepsilon)=1$, which gives $L=\mathcal{J}$ and $\nu(L-\mathcal{J},\varepsilon)=0$, which gives $L=\mathcal{J}$ for all $\varepsilon>0$. This yields to a contradiction. Therefore $K_2\subset K_1^c$. Since $x_{mn}\xrightarrow{\mathcal{I}_2^{(\mu,\nu)}}L$, then $K_1^c=\{(m,n)\in\mathbb{N}\times \mathbb{N}: \mu(x_{mn}-L,\frac{\varepsilon}{2})\leq 1-s\ \text{or}\ \nu(x_{mn}-L,\frac{\varepsilon}{2})\geq s \}\in\mathcal{I}_2$. Hence $K_2\in \mathcal{I}_2$, which contradicts $K_2\not\in\mathcal{I}_2$. Therefore, $\Gamma_{({x_{mn}})}(\mathcal{I}_2^{(\mu,\nu)})=\{L\}$. This completes the proof.
\end{proof}

\begin{theorem}
     Let $\{x_{mn}\}$ be a double sequence in an IFNS $(X,\mu,\nu,\star,\circ )$. Then, for all $r>0$ the set $\Gamma_{({x_{mn}})}^r(\mathcal{I}_2^{(\mu,\nu)})$ is closed with respect to the intuitionistic fuzzy norm $(\mu,\nu)$.
\end{theorem}

\begin{proof}
    The proof is almost similar to the proof of Theorem \ref{thm3.3}. So we omit details.
\end{proof}

\begin{theorem}\label{thm3.9}
    Let $\{x_{mn}\}$ be a double sequence in an IFNS $(X,\mu,\nu,\star,\circ )$. Then for an arbitrary $x_1\in \Gamma_{({x_{mn}})}(\mathcal{I}_2^{(\mu,\nu)})$ and $\lambda\in (0,1)$ we have $\mu(x_2-x_1,r)>1-\lambda$ and $\nu(x_2-x_1,r)<\lambda$ for all $x_2\in \Gamma_{({x_{mn}})}^r(\mathcal{I}_2^{(\mu,\nu)})$.
\end{theorem}

\begin{proof}
    For given $\lambda\in (0,1)$, choose $s\in (0,1)$ such that $(1-s)\star (1-s)>1-\lambda$ and $s\circ s<\lambda$. Since $x_1\in \Gamma_{({x_{mn}})}(\mathcal{I}_2^{(\mu,\nu)})$, for every $\varepsilon>0$, we get \begin{equation}\label{eqn3.2}
       \{(m,n)\in \mathbb{N}\times \mathbb{N}: \mu(x_{mn}-x_1,\varepsilon)>1-s \ \text{and}\ \nu(x_{mn}-x_1,\varepsilon)<s\}\notin\mathcal{I}_2. 
    \end{equation}
    Now, let $(i,j)\in \{(m,n)\in \mathbb{N}\times \mathbb{N}: \mu(x_{mn}-x_1,\varepsilon)>1-s \ \text{and}\ \nu(x_{mn}-x_1,\varepsilon)<s\}$. Then we have $\mu(x_{ij}-x_2,r+\varepsilon)\geq \mu(x_{ij}-x_1,\varepsilon)\star \mu(x_1-x_2,r)>(1-s)\star (1-s)>1-\lambda$ and $\nu(x_{ij}-x_2,r+\varepsilon)\leq \nu(x_{ij}-x_1,\varepsilon)\circ\nu(x_2-x_1,r)<s\circ s<\lambda$. Therefore $\{(m,n)\in \mathbb{N}\times \mathbb{N}: \mu(x_{mn}-x_1,\varepsilon)>1-s \ \text{and}\ \nu(x_{mn}-x_1,\varepsilon)<s\}\subset \{(m,n)\in \mathbb{N}\times \mathbb{N}: \mu(x_{mn}-x_2,r+\varepsilon)>1-s \ \text{and}\ \nu(x_{mn}-x_2,r+\varepsilon)<s\}$. So, from Equation \ref{eqn3.2}, we get $\{(m,n)\in \mathbb{N}\times \mathbb{N}: \mu(x_{mn}-x_2,r+\varepsilon)>1-s \ \text{and}\ \nu(x_{mn}-x_2,r+\varepsilon)<s\}\notin\mathcal{I}_2$. Hence $x_2\in \Gamma_{({x_{mn}})}^r(\mathcal{I}_2^{(\mu,\nu)})$. This completes the proof.
\end{proof}

\begin{theorem}\label{thm3.10}
    Let $\{x_{mn}\}$ be a double sequence in an IFNS $(X,\mu,\nu,\star,\circ )$. Then for some $r>0$, $\lambda\in(0,1)$ and fixed $x_0\in X$ we have $$\Gamma_{({x_{mn}})}^r(\mathcal{I}_2^{(\mu,\nu)})=\bigcup_{x_0\in \Gamma_{({x_{mn}})}(\mathcal{I}_2^{(\mu,\nu)})}\overline{B(x_0,\lambda,r)}.$$
\end{theorem}

\begin{proof}
    For a  given $\lambda\in (0,1)$, choose $s\in (0,1)$ such that $(1-s)\star (1-s)>1-\lambda$ and $s\circ s<\lambda$. Let $y_0\in \bigcup_{x_0\in \Gamma_{({x_{mn}})}(\mathcal{I}_2^{(\mu,\nu)})}\overline{B(x_0,\lambda,r)}$. Then there is $x_0\in \Gamma_{({x_{mn}})}(\mathcal{I}_2^{(\mu,\nu)})$ such that $\mu(x_0-y_0,r)>1-s$ and $\nu(x_0-y_0,r)<s$. Now, since $x_0\in \Gamma_{({x_{mn}})}(\mathcal{I}_2^{(\mu,\nu)})$, for every $\varepsilon>0$ there exists a set $M=\{(m,n)\in\mathbb{N}\times\mathbb{N}: \mu(x_{mn}-x_0,\varepsilon)>1-s \ \text{and}\ \nu(x_{mn}-x_0,\varepsilon)<s \}$ with $M\notin\mathcal{I}_2$. Let $(i,j)\in M$. Now we have $\mu(x_{ij}-y_0,r+\varepsilon)\geq \mu(x_{ij}-x_0,\varepsilon)\star\mu(x_0-y_0,r)>(1-s)\star (1-s)>1-\lambda$ and $\nu(x_{ij}-y_0,r+\varepsilon)\leq \nu(x_{ij}-x_0,\varepsilon)\circ\nu(x_0-y_0,r)<s\circ s<\lambda$. Therefore $M\subset \{(i,j)\in\mathbb{N}\times \mathbb{N}: \mu(x_{ij}-y_0,r+\varepsilon)>1-\lambda \ \text{and}\  \nu(x_{ij}-y_0,r+\varepsilon)<\lambda\}$. Since $M\notin \mathcal{I}_2$, $\{(i,j)\in\mathbb{N}\times \mathbb{N}: \mu(x_{ij}-y_0,r+\varepsilon)>1-\lambda \ \text{and}\  \nu(x_{ij}-y_0,r+\varepsilon)<\lambda\}\notin\mathcal{I}_2$. Hence $y_0\in \Gamma_{({x_{mn}})}^r(\mathcal{I}_2^{(\mu,\nu)})$. Therefore $\bigcup_{x_0\in \Gamma_{({x_{mn}})}(\mathcal{I}_2^{(\mu,\nu)})}\overline{B(x_0,\lambda,r)}\subseteq\Gamma_{({x_{mn}})}^r(\mathcal{I}_2^{(\mu,\nu)}) $.

   Conversely, suppose that $y_*\in \Gamma_{({x_{mn}})}^r(\mathcal{I}_2^{(\mu,\nu)})$. We shall show that $y_*\in \bigcup_{x_0\in \Gamma_{({x_{mn}})}(\mathcal{I}_2^{(\mu,\nu)})}\overline{B(x_0,\lambda,r)}$. If possible, let $y_*\notin \bigcup_{x_0\in \Gamma_{({x_{mn}})}(\mathcal{I}_2^{(\mu,\nu)})}\overline{B(x_0,\lambda,r)}$. So, $\mu(x_0-y_*,r)\leq 1-\lambda$ and $\nu(x_0-y_*,r)\geq \lambda$ for every $x_0\in \Gamma_{({x_{mn}})}(\mathcal{I}_2^{(\mu,\nu)})$. Now, by Theorem \ref{thm3.9}, we have $\mu(x_0-y_*,r)>1-\lambda$ and $\nu(x_0-y_*,r)<\lambda$, which is a contradiction. Therefore, $\Gamma_{({x_{mn}})}^r(\mathcal{I}_2^{(\mu,\nu)})\subseteq \bigcup_{x_0\in \Gamma_{({x_{mn}})}(\mathcal{I}_2^{(\mu,\nu)})}\overline{B(x_0,\lambda,r)}$. Hence $\Gamma_{({x_{mn}})}^r(\mathcal{I}_2^{(\mu,\nu)})=\bigcup_{x_0\in \Gamma_{({x_{mn}})}(\mathcal{I}_2^{(\mu,\nu)})}\overline{B(x_0,\lambda,r)}$. This completes the proof.
\end{proof}

\begin{theorem}\label{thm3.11}
    Let $\{x_{mn}\}$ be a double sequence in an IFNS $(X,\mu,\nu,\star,\circ )$. Then for any $\lambda\in (0,1)$ the following statements hold:
    \begin{enumerate}
        \item If $x_0\in \Gamma_{({x_{mn}})}(\mathcal{I}_2^{(\mu,\nu)})$ then $\mathcal{I}_2^{(\mu,\nu)}\text{-}LIM_{x_{mn}}^r\subseteq \overline{B(x_0,\lambda,r)}$.
        \item $\mathcal{I}_2^{(\mu,\nu)}\text{-}LIM_{x_{mn}}^r=\bigcap_{x_0\in \Gamma_{({x_{mn}})}(\mathcal{I}_2^{(\mu,\nu)})}\overline{B(x_0,\lambda,r)}=\{\eta\in X: \Gamma_{({x_{mn}})}(\mathcal{I}_2^{(\mu,\nu)})\subseteq \overline{B(\eta,\lambda,r)} \}$.
    \end{enumerate}
\end{theorem}

\begin{proof}
   \begin{enumerate}
       \item \label{Part1} For a given $\lambda\in (0,1)$, choose $s_1,s_2\in(0,1)$ such that $(1-s_1)\star (1-s_2)>1-\lambda$ and $s_1\circ s_2<\lambda$. If possible, we suppose that there exists an element $x_0\in \Gamma_{({x_{mn}})}(\mathcal{I}_2^{(\mu,\nu)})$ and $\gamma\in \mathcal{I}_2^{(\mu,\nu)}\text{-}LIM_{x_{mn}}^r$ such that $\gamma\notin \overline{B(x_0,\lambda,r)}$ i.e., $ \mu(\gamma-x_0,r)<1-\lambda$ and $\nu(\gamma-x_0,r)>\lambda$. Let $\varepsilon>0$ be given. Since $\gamma\in \mathcal{I}_2^{(\mu,\nu)}\text{-}LIM_{x_{mn}}^r$ , the sets $M_1=\{(m,n)\in\mathbb{N}\times\mathbb{N}: \mu(x_{mn}-x_0,\varepsilon)>1-s_1 \ \text{and}\ \nu(x_{mn}-x_0,\varepsilon)<s_1 \}\notin\mathcal{I}_2$ and $M_2=\{ (m,n)\in\mathbb{N}\times\mathbb{N}: \mu(x_{mn}-\gamma,r+\varepsilon)\leq 1-s_2 \ \text{or}\ \nu(x_{mn}-\gamma, r+\varepsilon)\geq s_2\}\in \mathcal{I}_2$. Now for $(i,j)\in M_1\cap M_2^c$, we have $\mu(\gamma-x_0,r)\geq \mu(x_{ij}-x_0,\varepsilon)\star\mu(x_{ij}-\gamma,r+\varepsilon)>(1-s_1)\star(1-s_2)>1-\lambda$ and $\nu(\gamma-x_0,r)\leq \nu(x_{ij}-x_0,\varepsilon)\circ \nu(x_{ij}-\gamma,r+\varepsilon)<s_1\circ s_2<\lambda$, which is a contradiction. Therefore $\gamma\in\overline{B(x_0,\lambda,r)}$. Hence $\mathcal{I}_2^{(\mu,\nu)}\text{-}LIM_{x_{mn}}^r\subseteq \overline{B(x_0,\lambda,r)}$.
       \item Using Part \ref{Part1}, we get \begin{equation}\label{eqn3.3}
           \mathcal{I}_2^{(\mu,\nu)}\text{-}LIM_{x_{mn}}^r\subseteq \bigcap_{x_0\in \Gamma_{({x_{mn}})}(\mathcal{I}_2^{(\mu,\nu)})}\overline{B(x_0,\lambda,r)}.
       \end{equation} Now, let $\beta\in \bigcap_{x_0\in \Gamma_{({x_{mn}})}(\mathcal{I}_2^{(\mu,\nu)})}\overline{B(x_0,\lambda,r)}$. So, we have $\mu(\beta-x_0,r)\geq 1-\lambda$ and $\nu(\beta-x_0,r)\leq \lambda$ for $x_0\in \Gamma_{({x_{mn}})}(\mathcal{I}_2^{(\mu,\nu)})$ and, therefore $\Gamma_{({x_{mn}})}(\mathcal{I}_2^{(\mu,\nu)})\subseteq \overline{B(\beta,\lambda,r)}$ i.e., we can write \begin{equation}\label{eqn3.4}
           \bigcap_{x_0\in \Gamma_{({x_{mn}})}(\mathcal{I}_2^{(\mu,\nu)})}\overline{B(x_0,\lambda,r)}\subseteq \{\eta\in X: \Gamma_{({x_{mn}})}(\mathcal{I}_2^{(\mu,\nu)})\subseteq \overline{B(\eta,\lambda,r)} \}
       \end{equation} Now we shall show that $\{\eta\in X: \Gamma_{({x_{mn}})}(\mathcal{I}_2^{(\mu,\nu)})\subseteq \overline{B(\eta,\lambda,r)} \}\subseteq \mathcal{I}_2^{(\mu,\nu)}\text{-}LIM_{x_{mn}}^r$. Let $\beta \notin \mathcal{I}_2^{(\mu,\nu)}\text{-}LIM_{x_{mn}}^r$. Then for $\varepsilon>0$, $\{(m,n)\in\mathbb{N}\times\mathbb{N}: \mu(x_{mn}-\beta,r+\varepsilon)\leq 1-\lambda \ \text{or}\ \nu(x_{mn}-\beta,r+\varepsilon)\geq \lambda \}\notin\mathcal{I}_2$, which gives there exists an $\mathcal{I}_2$-cluster point $x_0$ for the double sequence $\{x_{mn}\}$ with $\mu(\beta-x_0,r+\varepsilon)\leq 1-\lambda$ and $\beta-x_0,r+\varepsilon)\geq \lambda$. Hence $ \Gamma_{({x_{mn}})}(\mathcal{I}_2^{(\mu,\nu)})\nsubseteq \overline{B(\beta,\lambda,r)}$ and so, $\beta\notin \{\eta\in X: \Gamma_{({x_{mn}})}(\mathcal{I}_2^{(\mu,\nu)})\subseteq \overline{B(\eta,\lambda,r)} \}$. So, \begin{equation}\label{eqn3.5}
           \{\eta\in X: \Gamma_{({x_{mn}})}(\mathcal{I}_2^{(\mu,\nu)})\subseteq \overline{B(\eta,\lambda,r)} \}\subseteq \mathcal{I}_2^{(\mu,\nu)}\text{-}LIM_{x_{mn}}^r.
       \end{equation} Therefore from Equations \ref{eqn3.3}, \ref{eqn3.4} and \ref{eqn3.5}, we have $$\mathcal{I}_2^{(\mu,\nu)}\text{-}LIM_{x_{mn}}^r=\bigcap_{x_0\in \Gamma_{({x_{mn}})}(\mathcal{I}_2^{(\mu,\nu)})}\overline{B(x_0,\lambda,r)}=\{\eta\in X: \Gamma_{({x_{mn}})}(\mathcal{I}_2^{(\mu,\nu)})\subseteq \overline{B(\eta,\lambda,r)} \}.$$ This completes the proof.
   \end{enumerate}
\end{proof}

\begin{theorem}\label{thm3.12}
     Let $\{x_{mn}\}$ be a double sequence in an IFNS $(X,\mu,\nu,\star,\circ )$ such that $x_{mn}\xrightarrow{\mathcal{I}_2^{(\mu,\nu)}}y_*$. Then, for any $\lambda\in (0,1)$ and $r>0$, we have $\mathcal{I}_2^{(\mu,\nu)}\text{-}LIM_{x_{mn}}^r=\overline{B(y_*,\lambda,r)}$. 
\end{theorem}

\begin{proof}
  Let $\lambda_2\in (0,1)$. Choose $\lambda_1\in (0,1)$ such that $\lambda_1\star \lambda >\lambda_2$ and $\lambda_1\circ \lambda<\lambda_2$.  Let $\varepsilon>0$ be given. Since $x_{mn}\xrightarrow{\mathcal{I}_2^{(\mu,\nu)}}y_*$, the set $P= \{(m,n)\in \mathbb{N}\times \mathbb{N}: \mu(x_{mn}-y_*,\varepsilon)\leq 1-\lambda_1 \ \text{or}\ \nu(x_{mn}-y_*,\varepsilon)\geq \lambda_1\}\in \mathcal{I}_2$. Now, let  $\xi\in \overline{B(y_*,\lambda,r)}$. Then $\mu(\xi-y_*,r)\geq 1-\lambda$ and $\nu(\xi-y_*,r)\leq \lambda$. Now for $(m,n)\in P^c$, we have $\mu(x_{mn}-\xi,r+\varepsilon)\geq \mu(x_{mn}-y_*,\varepsilon)\star \mu(\xi-y_*,r)>(1-\lambda_1)\star(1-\lambda)>1-\lambda_2$ and $\nu(x_{mn}-\xi,r+\varepsilon)\leq \nu(x_{mn}-y_*,\varepsilon)\circ \nu(\xi-y_*,r)<\lambda_1\circ \lambda<\lambda_2$. Therefore $\{(m,n)\in \mathbb{N}\times \mathbb{N}: \mu(x_{mn}-\xi,r+\varepsilon)\leq 1-\lambda_2 \ \text{or}\  \nu(x_{mn}-\xi,r+\varepsilon)\geq \lambda_2\}\subset P$. Hence $\xi\in \mathcal{I}_2^{(\mu,\nu)}\text{-}LIM_{x_{mn}}^r$. Consequently, $\overline{B(y_*,\lambda,r)}\subseteq \mathcal{I}_2^{(\mu,\nu)}\text{-}LIM_{x_{mn}}^r$. Now, using Theorem \ref{thm3.7} and \ref{thm3.11}, we have $\mathcal{I}_2^{(\mu,\nu)}\text{-}LIM_{x_{mn}}^r\subseteq \overline{B(y_*,\lambda,r)}$. Therefore $\mathcal{I}_2^{(\mu,\nu)}\text{-}LIM_{x_{mn}}^r=\overline{B(y_*,\lambda,r)}$. This completes the proof.
\end{proof}

\begin{theorem}
    Let $\{x_{mn}\}$ be a double sequence in an IFNS $(X,\mu,\nu,\star,\circ )$ such that $x_{mn}\xrightarrow{\mathcal{I}_2^{(\mu,\nu)}}L$. Then, for any $\lambda\in (0,1)$,  $\Gamma_{({x_{mn}})}^r(\mathcal{I}_2^{(\mu,\nu)})=\mathcal{I}_2^{(\mu,\nu)}\text{-}LIM_{x_{mn}}^r$ for some $r>0$.
\end{theorem}

\begin{proof}
    Since $x_{mn}\xrightarrow{\mathcal{I}_2^{(\mu,\nu)}}L$, therefore from Theorem \ref{thm3.7}, $\Gamma_{({x_{mn}})}(\mathcal{I}_2^{(\mu,\nu)})=\{L\}$. Again from Theorem \ref{thm3.10}, $\Gamma_{({x_{mn}})}^r(\mathcal{I}_2^{(\mu,\nu)})=\overline{B(L,\lambda,r)}$. And, from Theorem \ref{thm3.12}, $\overline{B(L,\lambda,r)}=\mathcal{I}_2^{(\mu,\nu)}\text{-}LIM_{x_{mn}}^r$. Therefore $\Gamma_{({x_{mn}})}^r(\mathcal{I}_2^{(\mu,\nu)})=\mathcal{I}_2^{(\mu,\nu)}\text{-}LIM_{x_{mn}}^r$. This completes the proof.
\end{proof}

\subsection*{Acknowledgments}
 The second author is grateful to The Council of Scientific and Industrial Research (CSIR), HRDG, India, for the grant of Senior Research Fellowship during the preparation of this paper.


\begin{thebibliography}{99}\baselineskip=20pt
\footnotesize{
\bibitem{Atanassov}
K. Atanassov, Intuitionistic fuzzy sets, \textit{Fuzzy Sets and Systems}, \textbf{20} (1986), 87-96.

\bibitem{Antal}
R. Antal, M.  Chawla, V. Kumar,   Rough statistical convergence in intuitionistic fuzzy normed spaces, \textit{Filomat}, \textbf{35(13)} (2021), 4405-4416.

\bibitem{Ayter}
S. Aytar,   Rough statistical convergence, \textit{Numer. Funct. Anal. Optim.}, \textbf{29(3-4)} (2008), 291-303.

\bibitem{Barros}
L. C.  Barros,  R. C.  Bassanezi, P. A.  Tonelli,   Fuzzy modelling in population dynamics, \textit{Ecological modelling},  \textbf{128(1)} (2000), 27-33.

\bibitem{Banerjee}
A. K. Banerjee, R. Mondal,   Rough convergence of sequences in a cone metric space, \textit{ J.  Anal.}, \textbf{27 (4)} (2019), 1179-1188.

\bibitem{Das 2008}
P. Das,  P. Kostyrko, W. Wilczy\'{n}ski,  P. Malik,  $\mathcal{I}$ and $\mathcal{I^*}$-convergence of double sequences, \textit{ Math. Slovaca}, \textbf{58(5)} (2008), 605-620.

\bibitem{Dundar}
E. D\"{u}ndar,   On rough $\mathcal{I}_2$-convergence of double sequences, \textit{Numer. Funct. Anal.  Optim.}  \textbf{37 (4)} (2016), 480-491.


\bibitem{Esi}
A. Esi, N. Subramanian, A. Esi,   Wijsman rough $\mathcal{I}$-convergence limit point of triple sequences defined by a metric function, \textit{Ann.  Fuzzy Math.  Inform.}, \textbf{15 (1)} (2018), 47-57.

\bibitem{Fast}
H. Fast,  Sur la convergence statistique,  \textit{ Colloq. Math.} \textbf{ 2 (3-4)}, (1951), 241-244.

\bibitem{Fradkov}
A. L. Fradkov, R. J.  Evans,   Control of chaos: Methods and applications in engineering, \textit{Annu. Rev. Control}, \textbf{29(1)} (2005), 33-56.

\bibitem{Giles}
R. Giles,   A computer program for fuzzy reasoning,  \textit{Fuzzy Sets and Systems}, \textbf{4(3)} (1980), 221-234.

\bibitem{Ghosal}
S. Ghosal, A. Ghosh,   Rough weighted $\mathcal{I}$-limit points and weighted $\mathcal{I}$-cluster points in $\theta$-metric space, \textit{Math. Slovaca}, \textbf{ 70 (3)} (2020), 667-680.

\bibitem{Gumus}
H. G\"{u}m\"{u}\c{s}, N. Demir, Rough $\Delta\mathcal {I}-$ Convergence, \textit{Konuralp J. Math.}  \textbf{9 (1)} (2021), 209-216.

\bibitem{Hong}
L. Hong, J. Q.  Sun,   Bifurcations of fuzzy nonlinear dynamical systems, \textit{Commun. Nonlinear Sci. Numer. Simul.},  \textbf{11(1)}  (2006), 1-12.

\bibitem{Hossain2022}
N. Hossain, A. K. Banerjee,  Rough $\mathcal{I}$-convergence in  intuitionistic  fuzzy normed spaces, \textit{Bull. Math. Anal. Appl.} \textbf{14(4)} (2022), 1-10.


\bibitem{Kostyrko Salat Wilczynski}
P. Kostyrko, T. \v{S}al\'{a}t,  W.  Wilczy\'{n}ski, $\mathcal{I}$-convergence, \textit{Real Anal. Exchange}, \textbf{26(2)} (2000/01),   669-685.

\bibitem{Klement}
E. P. Klement, R. Mesiar, E. Pap, Triangular norms. Position paper I: basic analytical and algebraic properties, \textit{Fuzzy Sets and Systems}, \textbf{143} (2004), 5-26.

\bibitem{Kumar}
V. Kumar, M. Mursaleen,   On $(\lambda, \mu)$-statistical convergence of double sequences on intuitionistic fuzzy normed spaces, \textit{Filomat}, \textbf{25 (2)} (2011), 109-120.

\bibitem{Karakaya}
V. Karakaya, N. \c{S}im\c{s}ek, M. Ert\"{u}rk, F. G\"{u}rsoy,   On ideal convergence of sequences of functions in intuitionistic fuzzy normed spaces, \textit{Appl. Math. Inf. Sci.} \textbf{8 (5)} (2014), 2307-2313.

\bibitem{Konwar}
N. Konwar, P. Debnath,  $\mathcal{I}_\lambda$-convergence in intuitionistic fuzzy $n$-normed linear space, \textit{Ann. Fuzzy Math. Inform.} \textbf{13 (1)} (2017), 91-107.

\bibitem{Kirisci}
M. Kiri\c{s}ci,   Fibonacci statistical convergence on intuitionistic fuzzy normed spaces, \textit{Journal of Intelligent and Fuzzy Systems}, \textbf{36 (6)} (2019), 5597-5604.


\bibitem{Khan}
V. A. Khan, M.  Ahmad,   On $(\lambda, \mu)$-Zweier ideal convergence in intuitionistic fuzzy normed space, \textit{Yugosl. J.  Oper. Res.} \textbf{30 (4)} (2020), 413-427.

\bibitem{Kisi}
\"{O}. Ki\c{s}i, P. Debnath,   Fibonacci ideal convergence on intuitionistic fuzzy normed linear spaces, \textit{Fuzzy Information and Engineering}, \textbf{14 (3)}  (2022), 255-268.

\bibitem{Madore}
J. Madore,   Fuzzy physics,  \textit{Ann. Physics}, \textbf{219(1)} (1992), 187-198.

\bibitem{Mursaleen Edely}
M. Mursaleen, O. H. H. Edely, Statistical convergence of double sequences, \textit{J. Math. Anal. Appl.} \textbf{28(1)} (2003), 223-231.

\bibitem{Mursaleen 2009}
M. Mursaleen, S. A. Mohiuddine,   Statistical convergence of double sequences in intuitionistic fuzzy normed spaces, \textit{Chaos Solitons  Fractals}, \textbf{41 (5)} (2009), 2414-2421.

\bibitem{Mohiuddine}
S. A. Mohiuddine, Q. D. Lohani,   On generalized statistical convergence in intuitionistic fuzzy normed space, \textit{Chaos Solitons  Fractals}, \textit{42 (3)} (2009), 1731-1737.

\bibitem{Mursaleen 2010}
M. Mursaleen, S. A.  Mohiuddine, O. H. H.  Edely,   On the ideal convergence of double sequences in intuitionistic fuzzy normed spaces, \textit{Comput. Math.  Appl.}, \textbf{59(2)}  (2010), 603-611.

\bibitem{Ozcan}
A. \"{O}zcan, A. Or,  Rough statistical convergence of double sequences in intuitionistic fuzzy normed spaces, \textit{Journal of New Results in Science}, \textbf{11(3)} (2022), 233-246.

\bibitem{Phu2001}
H. X. Phu,   Rough convergence in normed linear spaces,  \textit{Numer. Funct. Anal. Optim.},  \textbf{22(1-2)} (2001), 199-222.

\bibitem{Phu2002}
H. X. Phu, Rough continuity of linear operators, \textit{Numer. Funct. Anal. Optim.} \textbf{23}, (2002), 139-146. 

\bibitem{Phu2003}
H. X. Phu, Rough convergence in infinite dimensional normed spaces, \textit{Numer. Funct. Anal. Optim.} \textbf{24},
(2003), 285-301. 

\bibitem{Park}
J. H. Park,   Intuitionistic fuzzy metric spaces, \textit{Chaos Solitons  Fractals}, \textbf{22(5)} (2004), 1039-1046.

\bibitem{Pal2013}
S. K. Pal,  D.  Chandra,  S. Dutta,   Rough ideal convergence, \textit{Hacet. J.  Math.  Stat.}, \textbf{42 (6)} (2013), 633-640.

\bibitem{Steinhaus}
H. Steinhaus,   Sur la convergence ordinaire et la convergence asymptotique, \textit{ Colloq. Math.} \textbf{ 2 (1)} (1951),  73-74.

\bibitem{Sklar}
B. Schweizer, A.  Sklar,   Statistical metric spaces, \textit{Pacific J. Math.}, \textbf{10(1)} (1960), 313-334.

\bibitem{Saadati 2006}
R. Saadati, J.H. Park,  On the intuitionistic fuzzy topological spaces,  \textit{Chaos Solitons  Fractals}, \textbf{27(2)} (2006), 331-344.

\bibitem{Saadati Park}
R. Saadati, J.H.  Park,   Intuitionistic fuzzy Euclidean normed spaces, \textit{Commun. Math. Anal.}, \textbf{1(2)} (2006), 85-90.

\bibitem{Sen}
M. Sen,   On $\mathcal{I}$-limit superior and $\mathcal{I}$-limit Inferior of sequences in intuitionistic fuzzy normed spaces, \textit{International Journal of Computer Applications}, \textbf{85(3)} (2014), 30-33.

\bibitem{Zadeh}
L. A. Zadeh, Fuzzy sets, \textit{Inform. control}, \textbf{8} (1965), 338-353.
}
\end{thebibliography}
\end{document}